\documentclass[reqno,11pt]{amsart}
\usepackage{a4wide,color,eucal,enumerate,mathrsfs}
\usepackage[normalem]{ulem}
\usepackage{amsmath,amssymb,epsfig,amsthm} 
\usepackage[latin1]{inputenc}
\usepackage{psfrag}
\usepackage{xcolor}
\definecolor{lightseagreen}{rgb}{0.13, 0.7, 0.67}
\definecolor{darkred}{rgb}{0.55, 0.0, 0.0}

\usepackage{pgfplots}
\pgfplotsset{compat=1.18}
\usepgfplotslibrary{polar}
\usetikzlibrary{intersections}
\usetikzlibrary{positioning}

\def\az{\alpha}

\def\dist{{\mathop\mathrm{\,dist\,}}}

\def\ez{\epsilon}

\def\bint{{\ifinner\rlap{\bf\kern.35em--}
\int\else\rlap{\bf\kern.45em--}\int\fi}\ignorespaces}

\def\bbint{{\ifinner\rlap{\bf\kern.35em--}
\hspace{0.078cm}\int\else\rlap{\bf\kern.45em--}\int\fi}\ignorespaces}

\newcommand{\R}{\mathbb{R}}

\newtheorem{thm}{Theorem}[section]
\newtheorem{lem}[thm]{Lemma}
\newtheorem{prop}[thm]{Proposition}
\numberwithin{equation}{section}

\theoremstyle{remark}
\newtheorem{rem}[thm]{Remark}

\def\bint{{\ifinner\rlap{\bf\kern.35em--}
\int\else\rlap{\bf\kern.45em--}\int\fi}\ignorespaces}

\usepackage{graphicx}
\usepackage{import}
\usepackage{xifthen}
\usepackage{pdfpages}
\usepackage{transparent}

\title[Serrin's overdetermined problem in rough domains]{Serrin's overdetermined problem in rough domains}
\author{Alessio Figalli and Yi Ru-Ya Zhang}
\date{\today}

\address{ETH Z\"urich, Department of Mathematics, R\"amistrasse 101, 8092, Z\"urich, Switzerland}
\email{alessio.figalli@math.ethz.ch}  
\address{State Key Laboratory of Mathematical Sciences, Academy of Mathematics and Systems Science, Chinese Academy of Sciences, Beijing 100190, China}
\address{Academy of Mathematics and Systems Science, the Chinese Academy of Sciences, Beijing 100190, China}
\email{yzhang@amss.ac.cn}

 \thanks{The second author is funded by National Key R\&D Program of China (Grant No. 2021YFA1003100), the Chinese Academy of Science,  and  NSFC grant No. 12288201. The first author have received funding from the European Research Council under the Grant Agreement No. 
721675 ``Regularity and Stability in Partial Differential Equations (RSPDE)''.}


\begin{document}
\begin{abstract}
The classical Serrin's overdetermined theorem states that a $C^2$ bounded domain, which admits a function with constant Laplacian that satisfies both constant Dirichlet and Neumann boundary conditions, must necessarily be a ball. While extensions of this theorem to non-smooth domains have been explored since the 1990s, the applicability of Serrin's theorem to Lipschitz domains remained unresolved. 

This paper answers this open question affirmatively. Actually, our approach shows that the result holds for domains that are sets of finite perimeter with a uniform upper bound on the density, and it also allows for slit discontinuities.
\end{abstract}
 
\maketitle

\section{Introduction}

Given $\Omega\subset\mathbb{R}^n$ a bounded domain, Serrin's overdetermined problem aims to understand how overdetermined problems for PDEs within $\Omega$ influence the geometry of the domain. In its simplest formulation, whenever $\partial\Omega$ is sufficiently smooth, one investigates the following problem:
\begin{equation}\label{serrin}
\Delta u = -1 \ \text{in} \ \Omega, \qquad u = 0 \ \text{on} \ \partial \Omega, \quad \partial_\nu u = \mathbf{c} > 0 \ \text{on} \ \partial \Omega.
\end{equation}
Here, $\partial_\nu u$ represents the inward normal derivative of $u$ on $\partial \Omega$. 

If $\Omega$ is not smooth but its boundary has at least finite $(n-1)$-dimensional Hausdorff measure, then \eqref{serrin} is understood in the following weak distributional sense: 
\begin{equation}\label{weak formulation}
u \in W^{1,2}_0(\Omega)\qquad\text{and}\qquad \int_{\Omega} \nabla u\cdot \nabla\varphi \, dx = -\mathbf{c}\int_{\partial \Omega} \varphi \, d\mathscr{H}^{n-1} + \int_{\Omega} \varphi \, dx\quad \forall\, \varphi\in C^1(\mathbb{R}^n),
\end{equation}
where $\mathscr H^{n-1}$ denotes the $(n-1)$-dimensional Hausdorff measure.

Given that the Dirichlet problem already yields a unique (weak) solution, the addition of the Neumann boundary condition makes the problem overdetermined. Consequently, \eqref{serrin} may not have a solution in general, which implies that the choice of domain $\Omega$ cannot be arbitrary.

\subsection{Serrin's Theorem}
In his seminal theorem, assuming that $\partial\Omega$ is of class $C^2$ (so that $u\in C^2(\Omega)\cap C^1(\overline{\Omega})$), Serrin proved the following celebrated result:

\begin{thm}[{\cite{S1971}}]\label{serrin thm}
Let $\Omega\subset \R^n$ be a $C^2$ bounded domain. Then \eqref{serrin} admits a solution $u \in C^2(\Omega)\cap C^1(\overline{\Omega})$ if and only if, up to a translation, $\Omega$ is a ball of radius $R=R(n,\mathbf{c})>0$ and $u$ takes the form
\begin{equation}\label{eq:sol u}
u(x) = \frac { R^2 - |x|^2 } { 2n }.
\end{equation}
\end{thm}

This revelation marked the beginning of a burgeoning and fertile area of mathematics, in which the interplay of analysis and geometry gave rise to many applications that span the most diverse areas of mathematics and the natural sciences. Remarkably, the genesis of this field can be traced back to a particular result that intriguingly arose from two questions in mathematical physics: one concerning the torsion of a straight solid rod, and the other concerning the tangential stress of a fluid on the walls of a rectilinear pipe; this was the original motivation of Serrin as stated in \cite{S1971}.

Serrin's proof builds on and refines the original concept introduced by Alexandrov \cite{A1958, A1962}, known today as the ``moving plane method''. Subsequently, Weinberger \cite{W1971} presented an alternative proof of Theorem~\ref{serrin thm}. Inspired by Weinberger's approach, researchers have explored alternative methods to establish this result, as evidenced in \cite{BNST2008, CH1998, PS1989}

\subsection{Generalizations}
Subsequently, Garofalo and Lewis demonstrated a similar result in \cite{GL1999} for the $p$-Laplacian. Then, these results were extended to operators in divergence form of $p$-Laplacian type, and even to certain special cases of $\infty$-Laplacian, in \cite{FK2008, FGK2006, BK2011, CS2009}. Moreover, equations involving fully nonlinear operators of non-divergence form, such as $k$-Hessian equations \cite{BNST2008}, and problems in space forms \cite{M1991, KP1998, CF2015, CV2017, QX2017, CV2019, FR2022, GJY2023, GMY2023}, have garnered significant attention. We also record some recent results in \cite{ABM2024} about a version of Serrin's problem on planar ring domains, where the solutions are not necessarily radially symmetric, except when adding further conditions on the number of critical points.

Given the extensive body of literature surrounding Serrin's original problem, we can only provide a glimpse of the breadth of results here. We suggest that interested readers look at the surveys \cite{S2001, NT2018} and the references therein for a comprehensive overview.

On a separate note, the proof in \cite{CH1998}, relying on Alexandrov's theorem, initially unveiled a connection between the results of Alexandrov and Serrin. Subsequently, a deeper linkage has been explored by \cite{CM2017, MP2019,MP2020}. We also recommend the insightful survey \cite{M2017} which comprehensively investigates these findings. This connection has also appeared in overdetermined elliptic problems within unbounded domains, initially conjectured by Berestycki, Caffarelli, and Nirenberg \cite{BCN1997} for balls and cylinders, and eventually disproved in \cite{S2010}. Subsequently, a multitude of counterexamples have been constructed based on unbounded constant mean curvature surfaces. Given the focus of our paper on bounded domains, we refer to the survey \cite{S2022} for further elucidation on related results.

\subsection{Serrin's Theorem for more singular domains}
In 1992, Vogel \cite{V1992} proved that if $\Omega$ is a $C^1$ domain for which a solution to \eqref{serrin} exists, then $\Omega$ is actually $C^{2}$ and therefore Theorem~\ref{serrin thm} holds. To be precise, Vogel assumed that
$$u(x)\to 0 \quad \text{ and } \quad  |\nabla u|(x)\to \mathbf{c} \ \text{ uniformly as } x\to \partial \Omega,$$
and then applied the regularity theory of free boundary problems of Alt-Caffarelli type. We note that his assumption is stronger than just assuming the validity of \eqref{weak formulation} (see also Remark~\ref{rem:bernoulli} below).  

Later, Berestycki posed the following question:\\
{\it Suppose $\Omega$ is $C^2$ throughout except for a potential corner, and $u$ represents a strong solution to \eqref{serrin} everywhere except at said corner. Does Serrin's Theorem remain applicable in this scenario?}\\
This problem was solved in \cite{P1998} using an adapted moving plane method, strategically circumventing the exceptional point.
Subsequently, interest arose regarding the extension of Serrin's theorem to more general domains. In particular, in \cite[Question 7.1]{HLL2024} it was asked:\\
{\it Does Theorem~\ref{serrin thm} hold if $\Omega$ is merely Lipschitz and $u$ solves
\eqref{weak formulation}?}

\subsection{Main result}
In this paper, we give a positive answer to the above question and we actually prove the validity of Theorem~\ref{serrin thm} to a much wider class of domains. To state our result we first observe that, since $u \in W^{1,2}_0(\Omega)$, we can extend $u$ to zero outside of $\Omega$ and  rewrite
$$
\int_\Omega \nabla u\cdot \nabla\varphi\,dx = 
\int_{\R^n} \nabla u\cdot \nabla\varphi\,dx=-\int_{\R^n} \Delta u\,\varphi\,dx,
$$
where $\Delta u$ denotes the distributional of $u$ on $\R^n$.
Hence \eqref{weak formulation} is equivalent to asking
$$
u \in W^{1,2}(\R^n),\qquad u=0\quad \text{a.e. in }\R^n\setminus \Omega,\qquad \Delta u=\mathbf{c}\mathscr{H}^{n-1}|_{\partial \Omega} - \mathbf{1}_{\Omega}\,dx,
$$
where the last equality should be intended in the sense of distribution.

Note that the formula above does not require $\Omega$ to be open but could be any Borel set, provided that we have a good notion of boundary that allows one to perform integration by parts.
This naturally leads to the notion of sets of finite perimeter, where $\partial\Omega$ should be replaced by the so-called ``reduced boundary'' $\partial^*\Omega$.
Moreover, we need a version of connectedness for sets of finite perimeter, called indecomposability.
We refer to Section~\ref{sec:prelim} below for more details.

\begin{rem}
We choose to work with sets of finite perimeters because the proof of Serrin's theorem would not be significantly easier if we assumed \(\Omega\) to be a Lipschitz domain; the main concepts introduced in this paper would still be necessary. For readers who are not concerned with this level of generality, we suggest reading our paper with the assumption that \(\Omega\) is a Lipschitz domain, so that \(\partial^*\Omega\) corresponds to the set of points where the boundary is differentiable (which is true \(\mathscr{H}^{n-1}\)-a.e. by Rademacher's theorem).
\end{rem}

We can now state our main theorem.

\begin{thm}\label{main thm}
Let $\Omega\subset \mathbb R^n$ be a bounded indecomposable set of finite perimeter satisfying
\begin{equation}\label{domain boundary1}
\mathscr H^{n-1}(B_r(x)\cap \partial^* \Omega)\le A r^{n-1} \quad \text{ for $\mathscr H^{n-1}$-a.e. } x\in \partial^* \Omega \text{ and $r \in (0,1)$,} 
\end{equation}
for some constant $A>0$.
 Then $\Omega$ admits a solution $u\in W^{1,2}(\R^n)$ to 
\begin{equation}
\label{eq:weak set finite per}
u=0\quad \text{a.e. in }\R^n\setminus \Omega,\qquad \Delta u=\mathbf{c}\mathscr{H}^{n-1}|_{\partial^*\Omega} - \mathbf{1}_{\Omega}\,dx,
\end{equation}
 if and only if, up to a translation, 
 $\Omega$ is a ball of radius $R=R(n,\mathbf{c})$ and $u$ is given by \eqref{eq:sol u}. 
\end{thm}

\begin{rem}
Theorem~\ref{main thm} implies, in particular, the validity of Serrin's theorem for any domain $\Omega$ whose boundary $\partial\Omega$ satisfies \eqref{domain boundary1}.
Assumption \eqref{domain boundary1} is notable mild and encompasses many geometries, including Lipschitz domains, and it also allows for the presence of countably many cusps.
Consequently, Theorem~\ref{main thm} not only addresses \cite[Question 7.1]{HLL2024} but also includes all previously established results.
\end{rem}


\begin{rem}
\label{rem:bernoulli}
Our theorem closely aligns with the result in \cite{DM2019} regarding the validity of Alexandrov's theorem for sets of finite perimeter. However, despite the similarity, the two problems are distinct. In particular, while our theorem is non-trivial even when
$\Omega$ is a Lipschitz domain, the validity of Alexandrov's theorem for Lipschitz domains follows directly from elliptic regularity theory (which immediately implies that such domains must be smooth).

To elucidate this point, it is important to note that Serrin's problem is related to the one-phase Bernoulli problem (we refer to the recent monograph \cite{V2023} for more details). Consequently, one might consider leveraging its regularity theory to demonstrate that solutions to Serrin's problem in rough domains are smooth. Unfortunately, this regularity theory relies either on a minimality property or a viscosity-type approach. In particular, for the distributional formulation \eqref{eq:weak set finite per}, it is unclear to us if a priori there is any comparison principle valid, and we proved it by showing that the harmonic measure associated to $\Omega$ is absolutely continuous with respect to $\mathscr H^{n-1}|_{\partial^* \Omega}$ in Lemma~\ref{greens function}. 
Hence, it is currently unclear to us whether it can be applied in our context, even in the case where $\Omega$ is Lipschitz. 
Therefore, our proof utilizes techniques from geometric measure theory.
\end{rem}

\begin{rem}
\label{rem:slit}
One may wonder whether Theorem~\ref{serrin thm} also holds in the case of sets with slit discontinuities (usually called ``slit domains'', for example, a slit ball. These domains cannot be treated in the framework of sets of finite perimeters, since the latter are defined up to sets of measure zero. Moreover, in the presence of slits, one must properly define in which sense \eqref{serrin}
is satisfied. As we shall see in Section~\ref{sect:extension}, our method can be adapted to prove Serrin's theorem in this more general setting.
\end{rem}

\begin{rem}
Recently, the works  \cite{BF20231, BF20232} extended the method of moving planes to general measurable sets. However, that approach does not seem to readily adapt to our specific setting. Specifically, in Serrin's work, the application of the moving planes method relies on the assumption that the normal derivative is constant. In the weak setting \eqref{eq:weak set finite per}, this information is only available in an integral sense or at points of the reduced boundary, which complicates the application of this method.\\
Even more, consider for example the two-dimensional set shown in grey in the figure below and given, in polar coordinates, by the formula   $\Omega=\left\{ (r,\theta) \,:\, \frac{6}{5} - |\cos(2\theta)|^{1/4} \leq r \leq \frac{3}{2} \right\}$.

\noindent

\begin{minipage}{0.37\textwidth}

\vspace{0.1cm}
\begin{center}
  \includegraphics[width=0.6\textwidth]{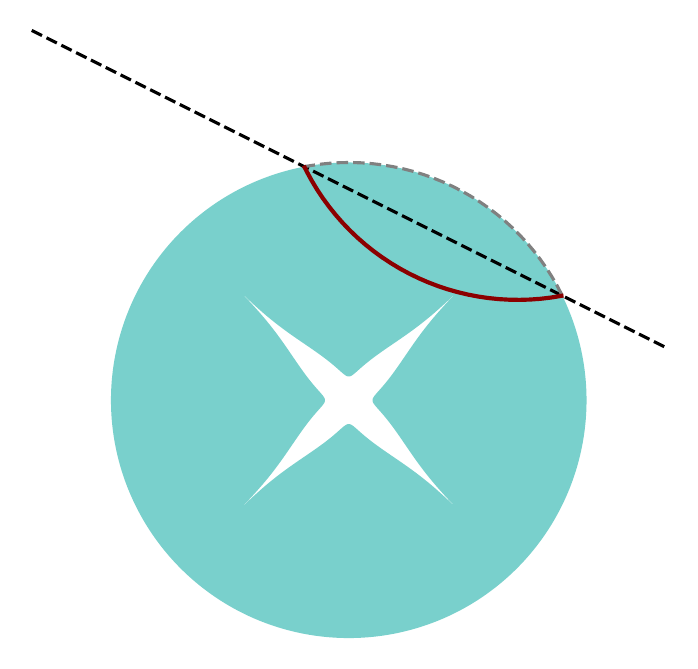}
  
\end{center}

\end{minipage}%
\begin{minipage}{0.6\textwidth}
\vspace{-0.2cm}
Applying the moving plane method to $\Omega$ leads to a reflection of the outer boundary that touches the inner boundary at one of the cusp points, where the Neumann boundary condition is not defined. This highlights a limitation of the moving plane method, which our approach overcomes. In particular, Theorem \ref{main thm} shows that such a domain $\Omega$ cannot be a solution to Serrin's problem.
\end{minipage}

\end{rem}

\medskip

The paper is structured as follows: In the next section, we collect some preliminary results that will be used in Section~\ref{sec:proof} to prove Theorem~\ref{main thm}. Then, Section~\ref{sect:extension} is devoted to the proof of Serrin's Theorem in slit domains, see Theorem~\ref{main thm 2} below.

\medskip

{\noindent \it Acknowledgments.}
The authors would like to thank Joaquim Serra for several discussions on this problem and Mingxuan Yang for feedback on an earlier version. The second author would like to express his gratitude to the Forschungsinstitut f\"ur Mathematik (FIM) for the warm hospitality and support during his visit, where this work was completed.


\section{Preliminary results}
\label{sec:prelim} 

In this section, we establish certain regularity results for $u$ and $\Omega$ and prove a volume identity that will be used in the proof of our main theorem. 
Before that, we recall the definition of sets of finite perimeter and their main properties.

    A measurable set $\Omega\subseteq \R^n$ is a set of \emph{finite perimeter} if the distributional gradient of its characteristics function $\mathbf{1}_\Omega$ is a $\R^n$-valued Radon measure $D\mathbf{1}_\Omega$ with finite total variation, i.e., $|D\mathbf{1}_\Omega|(\R^n)<\infty$.
It follows from the Lebesgue-Besicovitch theorem on differentiation of measures that for $|D\mathbf{1}|$-a.e. $x$, it holds
\begin{equation}\label{cond}
\lim_{r\to 0^+}\frac{D\mathbf{1}_\Omega(x+rB^n)}{|D\mathbf{1}_\Omega|(x+rB^n)} = \nu_x \quad \text{and} \quad |\nu_x| = 1.
\end{equation}
The set of points $x$ such that \eqref{cond} holds is called the \emph{reduced boundary of} $\Omega$ and denoted by $\partial^*\Omega$.
Also, at points of the reduced boundary, $\nu_x$  is the {\it measure-theoretic inner unit normal} to $\Omega$ at $x$.
According to De Giorgi Rectifiability Theorem, the reduced boundary is a $(n-1)$-rectifiable set.
Also, up to changing $\Omega$ in a set of measure zero, one can assume that $\overline{\partial^* \Omega}=\partial\Omega$.
Finally, a set of finite perimeter $E$ is said \emph{indecomposable} if for every $ F \subset  E$ having finite
 perimeter and such that
 $$\mathscr H^{n-1}(\partial^* E) = \mathscr H^{n-1}(\partial^* F) + \mathscr H^{n-1}(\partial^* (E\setminus F)), $$
 one has either $|F|=0$ or $|E\setminus F|=0$. 
We refer the interested reader to \cite{ACMM2001} and \cite[Sections 12 and 15]{M2012} for more details on sets of finite perimeter.

In the next lemma, $\mathring\Omega$ denotes the (topological) interior part of $\Omega$.

\begin{lem}\label{basic property u}
    Let $\Omega\subset \mathbb R^n$ be a bounded set of finite perimeter satisfying \eqref{domain boundary1},
    and let $u \in W^{1,2}(\R^n)$ satisfy \eqref{eq:weak set finite per}.
Then:
\begin{enumerate}
    \item[(1)] $u$ is $L$-Lipschitz continuous, with $L=L(n,\mathbf{c},A)$.
    \item[(2)] $u$ is nonnegative,  $\mathring\Omega=\{u>0\}$, and $u \in C^\infty(\mathring\Omega)$.
     \item[(3)] $\Omega=\mathring\Omega$ up to a set of measure zero. In particular, without loss of generality, we can assume $\Omega$ to be open.
\item[(4)] At every point $x\in \partial^*\Omega$ it holds
$$
\frac{u(x+rz)}{r} \to \mathbf{c}\big(\nu_x\cdot z \big)_+ \quad \text{ as }r\to 0,
$$
where $\nu_x$ denotes the measure-theoretic inner unit normal at $x$. 
\end{enumerate}
\end{lem}
\begin{proof}
Equation \eqref{eq:weak set finite per} implies that $\Delta u$ is a Radon measure. Also, it follows from \eqref{domain boundary1} that, given $x \in \partial^* \Omega$,
$$
\Delta u(B_r(x))\leq \big(\Delta u+\mathbf{1}_{\Omega}\big)(B_r(x))\le A\mathbf{c} r^{n-1} \quad \forall\,r \in (0,1).
$$
Since $u=0$ a.e. outside $\Omega$,  the classical identity
\begin{equation}
\label{eq:d dr u Delta}
\frac{d}{dr}\bint_{\partial B_r(x)} u\, d\mathscr H^{n-1}=\frac{\Delta u(B_r(x))}{n|B_1|r^{n-1}}
\end{equation}
(see for instance the proof of \cite[Lemma 3.10]{V2023})
combined with the bound above 
implies that 
$$\bint_{\partial B_r(x)} u\, d\mathscr H^{n-1}\le   C(n,A,\mathbf{c})  r \qquad \text{for }\,x \in \partial^* \Omega,\,r \in (0,1),$$
therefore
$$\bint_{B_r(x)} u(y)\, dy\le   C(n,A,\mathbf{c})  r \qquad \text{for }\,x \in \partial^* \Omega,\,r \in (0,1).$$
Since the map $x \mapsto \bbint_{B_r(x)} u(y)\, dy$ is continuous for $r>0$ fixed and $\overline{\partial^* \Omega}=\partial\Omega$, we deduce that
$$\bint_{B_r(x)} u(y)\, dy\le   C(n,A,\mathbf{c})  r \qquad \text{for }\,x \in \partial \Omega,\,r \in (0,1).$$
Recalling that $\Delta u=-1$ inside $\mathring\Omega$, interior regularity estimates imply that $u$ is uniformly Lipschitz continuous function inside $\mathring\Omega$,
see for instance the proof of \cite[Lemma 3.5]{V2023}.
Therefore, since $u$ vanishes on $\partial\Omega$ and $u=0$ a.e. outside $\Omega$, $u$ is globally Lipschitz, which proves (1).

Since $\Delta u=-1<0$ inside $\mathring \Omega$, the strong maximum principle implies that the set $\{u>0\}$ is inside $\mathring\Omega$. Since $u$ vanishes outside $\mathring\Omega$, this shows that $\{u>0\}=\mathring\Omega$. The smoothness of $u$ inside $\mathring\Omega$ follows immediately from the equation $\Delta u=-1.$ This proves (2).

Note that, since $u \geq 0$, the distributional Laplacian of $u$ (which we know to be a Radon measure) is non-negative in the set $\{u=0\}=\R^n\setminus \mathring\Omega$. Also, inside $\mathring\Omega=\{u>0\}$, $\Delta u=-1$. This implies that $$\Delta u= \mu_+- \mathbf{1}_{\mathring\Omega}\,dx,
$$ for some nonnegative measure $\mu_+$. Comparing this equation with  \eqref{eq:weak set finite per}, we conclude that $\mu_+=\mathbf{c}\mathscr{H}^{n-1}|_{\partial^*\Omega}$ and $\mathbf{1}_{\Omega}\,dx=\mathbf{1}_{\mathring\Omega}\,dx$. This last equality implies that $\Omega$ and $\mathring\Omega$ coincide a.e., proving (3).


Finally, we prove (4). Given $x\in \partial^*\Omega$, we consider the sequence of Lipschitz functions
$$
v_r(z)=\frac{u(x+rz)}{r},\qquad \Delta v_r=\mathbf{c}\mathscr{H}^{n-1}|_{\partial^* \Omega_{x,r}}-r\mathbf{1}_{\Omega_{x,r}}dx,
$$
where $\Omega_{x,r}=\frac{\Omega-x}{r}=\{z \in \R^n\,:\,x+rz \in \Omega\}$.
Then, since $x$ belongs to  $\partial^*\Omega$, $\Omega_{x,r} \to H_x$, where $H_x$ is a half space.
In addition, for any converging subsequence $r_i \to 0$, the Lipschitz functions $v_{r_i}$ 
converge to a nonnegative Lipschitz function $v_0$ satisfying
$$
\Delta v_0=0\quad \text{in }H_x,\qquad  v_0=0\quad \text{in }\R^n\setminus H_x.
$$
Then, Liouville Theorem implies that $v_0(z)=\mathbf{a}(\nu_x\cdot z)_+$ for some $\mathbf{a} \geq 0$.

On the other hand, using again $x\in \partial^*\Omega$, it follows that $\mathscr{H}^{n-1}|_{\partial^* \Omega_{x,r}} \rightharpoonup \mathscr{H}^{n-1}|_{\partial H_x}$.
Hence $\Delta v_0=\mathbf{c}\mathscr{H}^{n-1}|_{\partial H_x}$. Combining these two facts, we conclude that $\mathbf{a}=\mathbf{c}$.
This shows that $v_{r_i}\to \mathbf{c}(\nu_x\cdot z)_+$ for any converging subsequence $r_i$, so the entire sequence $v_r$ converges to $\mathbf{c}(\nu_x\cdot z)_+$, as desired.


\end{proof}

We now demonstrate the following volume identity that will play a crucial role in our proof. In the classical setting \cite{W1971}, this identity is proved via a Pohozaev's approach. Unfortunately, this approach requires too much regularity on the solution $u$, so a new proof is needed.

\begin{lem}\label{lem:volume}
Let $\Omega\subset \mathbb R^n$ be a bounded set of finite perimeter satisfying \eqref{domain boundary1},
    and let $u \in W^{1,2}(\R^n)$ satisfy \eqref{eq:weak set finite per}. Then 
\begin{equation}\label{volume}
    (n+2)\int_{\Omega} u\, dx =\mathbf{c}^2 n |\Omega|. 
\end{equation}
\end{lem}
\begin{proof}
Thanks to Lemma~\ref{basic property u}(3), we can assume that $\Omega$ is open.
    Given $\ez>0$, consider
\begin{equation}
    \label{eq:varphi eps}
\varphi_\ez(x)= \frac{u((1+\ez) x)-u((1-\ez)x)}{2\ez} - 2u(x).
\end{equation}
Note that $\varphi_\ez$ is Lipschitz continuous for $\ez>0$ fixed, and that $|\varphi_\ez|\leq C=C(L,\Omega)$ for all $\ez>0$ (by the Lipschitz continuity of $u$). 
Therefore, testing \eqref{eq:weak set finite per} against $\varphi_\ez$ we get
\begin{equation}
    \label{eq:test vol}
\mathbf{c} \int_{\partial^*\Omega}  \varphi_\ez \, d\mathscr H^{n-1} -\int_{\Omega} \varphi_\ez\, dx
= - \int_{\Omega} \nabla u\cdot \nabla\varphi_\ez \, dx =  \int_{\Omega} u \Delta  \varphi_\ez \, dx.  
\end{equation}
We first want to compute the Laplacian of $ \varphi_\ez$. To this aim, we define
$$\mathcal{N}_\ez:=\left\{ x\in \R^n\colon \dist(x,\partial \Omega)\leq C_0\ez \right\}$$
with $C_0={\rm diam}(\Omega)$, so that
$$
x,(1+\ez)x,(1-\ez)x \in \Omega \qquad \text{for all }x \in \Omega\setminus \mathcal N_\ez.
$$
Hence, noticing that $\frac{(1+\ez)^2-(1-\ez)^2}{2\ez}-2=0$, it follows from
\eqref{eq:weak set finite per} that
\begin{multline*}
    \Delta   \varphi_\ez(x)=\mathbf{c} \frac{(1+\ez)^2\mathscr H^{n-1}|_{(1+\ez)^{-1}\partial^*\Omega} - (1-\ez)^2\mathscr H^{n-1}|_{(1-\ez)^{-1}\partial^*\Omega}}{2\ez} \\
    -2\mathbf{c}\mathscr H^{n-1}|_{\partial^*\Omega} +O\biggl(\frac {1} {\ez} dx  |_{\mathcal N_\ez} \biggr).
\end{multline*} 
Therefore, by a change of variables,
\begin{equation}
\begin{split}
    &\int_\Omega u \Delta   \varphi_\ez\,dx\\
    &=
    \frac{\mathbf{c} }{2\ez}\left[\int_{(1+\ez)^{-1}\partial^*\Omega}  (1+\ez)^2 u \,d \mathscr H^{n-1} - \int_{(1-\ez)^{-1}\partial^*\Omega}    (1-\ez)^2 u\,d \mathscr H^{n-1} \right]+\int_{\mathcal N_\ez} u \,O\biggl(\frac {1} {\ez}\biggr) \, dx\\
&=\mathbf{c}\left[\int_{ \partial^*\Omega}   \frac{(1+\ez)^{3-n} u((1+\ez)^{-1} x) -(1-\ez)^{3-n}  u( (1-\ez)^{-1} x)}{2\ez} \,d \mathscr H^{n-1}(x)\right]+\int_{\mathcal N_\ez} u\, O\biggl(\frac {1} {\ez}\biggr) \, dx.\label{varphi eps}
    \end{split}
\end{equation} 
Hence, if we define
\begin{equation}
    \label{eq:psi eps}
\psi_\ez(x)=-\frac{(1+\ez)^{3-n} u((1+\ez)^{-1} x) -(1-\ez)^{3-n}  u( (1-\ez)^{-1} x)}{2\ez}+\varphi_\ez(x),
\end{equation}
it follows from \eqref{eq:test vol} that
\begin{equation}
    \label{eq:test vol2}
\mathbf{c} \int_{\partial^*\Omega}  \psi_\ez \, d\mathscr H^{n-1} -\int_{\Omega} \varphi_\ez\, dx
=  \int_{\mathcal N_\ez} u\, O\biggl(\frac {1} {\ez}\biggr) \, dx.
\end{equation}
Note now that, applying Lemma~\ref{basic property u}(4),  for $x\in \partial^*\Omega$ we have
$$
\psi_\ez(x) 
\to - \mathbf{c} \big((\nu_x\cdot x)_+ - (\nu_x\cdot x)_-\big)=-\mathbf{c}\nu_x\cdot x.
$$
Thus, by dominated convergence (recall that $|\varphi_\ez|\leq C$)
$$
\mathbf{c} \int_{\partial^*\Omega}  \psi_\ez \, d\mathscr H^{n-1} \to 
-\mathbf{c}^2 \int_{\partial^*\Omega} \nu_x\cdot x \, d\mathscr H^{n-1}\qquad \text{as }\ez \to 0.
$$
Also, since $\varphi_\ez(x) \to \nabla u(x)\cdot x-2u(x)$ inside $\Omega$ (recall that, without loss of generality, we can assume that $\Omega$ is open), by dominated convergence and an integration by parts  we have
$$
\int_{\Omega} \varphi_\ez\, dx \to \int_{\Omega} \big(\nabla u\cdot x-2u\big)\, dx=-(n+2)\int_{\Omega} u\, dx\qquad \text{as }\ez \to 0.
$$
Finally, since $u$ is Lipschitz and vanished on $\partial\Omega$ we deduce that 
$u=O(\ez)$ inside $\mathcal N_\ez$,
therefore
$$\int_{\mathcal N_\ez} u \,O\biggl(\frac {1} {\ez}\biggr) \, dx=O(|\mathcal N_\ez|)\to 0\qquad \text{as }\ez \to 0$$
(the convergence $|\mathcal N_\ez|\to 0$ is a simple consequence of dominated convergence, since $\mathcal N_\ez\to \emptyset$ as $\ez \to 0$).
Combining all together, we proved that
$$
-\mathbf{c}^2 \int_{\partial^*\Omega} \nu_x\cdot x \, d\mathscr H^{n-1}=(n+2)\int_\Omega u\,dx.
$$
Finally, by the divergence theorem for sets of finite perimeter (recall that $\nu_x$ is the inner normal),
$$  -  \int_{\partial^*\Omega}  \nu_x \cdot x \, d\mathscr H^{n-1}=  \int_{\Omega}{\rm div}(x)\,dx=n|\Omega|,$$
which concludes the proof.
\end{proof}

\section{Proof of Theorem~\ref{main thm}}
\label{sec:proof}
Our proof of Theorem~\ref{main thm} relies on Proposition~\ref{max nabla u} below, stating that $|\nabla u|\le \mathbf{c}$ in $\Omega$. 
Proving this fact is nontrivial for two reasons:\\
(i) in a rough domain as in our setting, the standard maximum principle for $|\nabla u|$ is not available;\\
(ii) in our situation, at least formally, Lemma~\ref{basic property u}(4) tells us that $|\nabla u|=\mathbf{c}$ on the reduced boundary, but we do not have any information at points of $\partial\Omega\setminus \partial^*\Omega$. \\
To circumvent these difficulties and to prove Proposition~\ref{max nabla u}, we shall carefully exploit the properties of the Green function of $\Omega$. As we shall see, to prove Lemma~\ref{greens function} we will crucially use the fact that a solution to Serrin's problem exists in $\Omega.$

\begin{rem} A Green function approach to prove a maximum principle on $|\nabla u|$ has already been used, in a similar context, in \cite[Theorem 6.3]{AC1981}. There, however, the authors assume that solutions are non-degenerate, which is something that we do not have in our context. For this reason, our proofs are completely different.
\end{rem}

Recall that, due to Lemma~\ref{basic property u}(3), we can assume that $\Omega$ is open.
To define the Green function, we consider an increasing sequence $\Omega_k$ of smooth sets contained inside $\Omega$ such that $\Omega_k \to \Omega$ as $k \to \infty$.
Then, given $x \in \Omega$, we note that $x \in \Omega_k$ for $k$ sufficiently large, so we can define $G_{x,k}$ as the solution of
$$
\left\{
\begin{array}{ll}
\Delta G_{x,k}=-\delta_{x},& \text{in }\Omega_k\\
G_{x,k}=0& \text{in }\R^n\setminus \Omega_k
\end{array}.
\right.
$$
Noticing that $G_{x,k}\leq G_{x,j}$ for $k<j$ (by the maximum principle), we can define 
$$G_x:=\lim_{k\to \infty}G_{x,k}.$$ 

\begin{lem}\label{greens function}
    Let $\Omega$ be an open bounded set satisfying the assumptions in Theorem~\ref{main thm}, and let a solution $u$ of \eqref{eq:weak set finite per} exist. Fix $x\in \Omega$, and let $G_x$ be the Green function constructed above. Then:
    \begin{enumerate}
    \item[(1)] $G_x$ is Lipschitz continuous near $\partial \Omega$, with the Lipschitz constant depending only on $n$, $A$, $\mathbf{c}$, $\Omega$, and $x$. Moreover, there exists a bounded measurable function $\az:\partial^*\Omega \to [0,\infty)$ such that
    $$\Delta G_x= \az\,\mathscr H^{n-1}|_{\partial^*\Omega}-\delta_x.$$     
    \item[(2)]  At every point $y\in \partial^*\Omega$ it holds
$$
\frac{G_x(y+rz)}{r} \to \mathbf{a}_y\big(\nu_y\cdot z \big)_+ \quad \text{ as }r\to 0,
$$
where $\nu_y$ denotes the measure-theoretic inner unit normal at $y$ and $\mathbf{a}_y=\az(y)$. 
    \end{enumerate}
\end{lem}
\begin{proof}
Recall that $G_x$ is the monotone limit of $G_{x,k}$.
As  $u>0$ in $\Omega$,
we can choose $\rho=\rho(x,\Omega)>0$ small and $M=M(x,\Omega)>0$ large enough (independent of $k$), so that 
$$G_{x,k}<Mu \quad  \text{ on } \  \partial B_{\rho}(x)$$
for all $k \gg 1$.
Also, as $u>0$ on $\partial\Omega_k$ we have $Mu> G_{x,k}$ on $\partial \Omega_k$, therefore 
$$  G_{x,k}<Mu  \quad \text{ in } \ \Omega_k\setminus B_\rho(x)$$ 
by the maximum principle (recall that $\Delta u=-1<0=\Delta G_{x,k}$ inside $\Omega_k\setminus B_\rho(x)$).
Letting $k\to \infty$, 
this gives
\begin{equation}
    \label{|eq:G u}
 G_{x}\leq Mu  \quad \text{ in } \ \Omega\setminus B_\rho(x).
\end{equation}
Recalling that $u$ vanishes on $\partial\Omega$ and it is Lipschitz, this proves that $G_x$
grows at most linearly at every boundary point. So, by interior regularity estimates, we conclude that $G_x$ is Lipschitz continuous in a neighborhood of $\partial \Omega$

Recalling that $G_{x,k}=0$ outside $\Omega_k$ and that $\Delta G_{x,k}+\delta_x \geq 0$, we see that 
 $G_x = 0$ outside $\Omega$ and that $\Delta G_x+\delta_x \geq 0$ as a distribution. In addition, $\Delta G_x=0$ inside $\Omega\setminus \{x\}$.

Now, given $y \in \R^n\setminus \Omega$ and $r>0$ small, it follows from \eqref{eq:d dr u Delta}, \eqref{|eq:G u}, and \eqref{eq:weak set finite per} that
$$
\int_0^r \frac{\Delta G_x(B_s(y))}{s^{n-1}} \,ds \leq M\int_0^r \frac{\Delta u(B_s(y))}{s^{n-1}}\,ds=M\mathbf{c}\int_0^r \frac{\mathscr H^{n-1}(\partial^*\Omega\cap B_s(y))}{s^{n-1}}\,ds.
$$
Hence, by the rectifiability of the reduced boundary,
$$
\liminf_{s \to 0}\frac{\Delta G_x(B_s(y))}{s^{n-1}} \leq M\mathbf{c}\lim_{s \to 0}\frac{\mathscr H^{n-1}(\partial^*\Omega\cap B_s(y))}{s^{n-1}}=
\left\{
\begin{array}{cl}
M\mathbf{c}\omega_{n-1}& \text{for $\mathscr H^{n-1}$-a.e. $y \in \partial^*\Omega$}\\
0& \text{for $\mathscr H^{n-1}$-a.e. $y \not\in \partial^*\Omega$},
\end{array}
\right.
$$
where $\omega_{n-1}$ denotes the volume of the $(n-1)$-dimensional ball.
By \cite[Theorem 6.11]{M1995}, this implies that
$$
\Delta G_x|_{\R^n\setminus \Omega} \leq C(M,\mathbf{c},n)\mathscr P^{n-1}|_{\partial^*\Omega},
$$
where $\mathscr P^{n-1}$ denotes the $(n-1)$-dimensional packing measure (see \cite[Chapter 5.10]{M1995}). Since $\partial^*\Omega$ is rectifiable it follows that $\mathscr P^{n-1}|_{\partial^*\Omega}=\mathscr H^{n-1}|_{\partial^*\Omega}$.

Thus, combining all together, we proved that
$$
0 \leq \Delta G_x+\delta_x \leq C(M,\mathbf{c},n) \mathscr H^{n-1}|_{\partial^*\Omega}.
$$
By the Radon-Nikodym Theorem, this implies the existence of a measurable function $\az :\partial^*\Omega\to \mathbb [0,C(M,n)]$ such that 
$$\Delta G_x =\az\mathscr H^{n-1}|_{\partial^*\Omega} -\delta_x.$$
This concludes the proof of (1).

The proof of (2) is similar to that of Lemma~\ref{basic property u}(4).  Indeed, given $y\in \partial^*\Omega$, it follows by (1) that the sequence of functions
$$F_r(z)=\frac{G_x(y+rz)}{r}$$
are uniformly Lipschitz and harmonic inside $U_{y,r}=\frac{U-y}{r}$, where $U=\Omega\setminus B_{\rho}(x). $ Also, since $y\in \partial^*\Omega$,
$$U_{y,r}\to H_y \quad \text{ as }\  r\to 0,$$
and, up to a subsequence, $F_r$ converges to a nonnegative Lipschitz function $F_0$ that is harmonic in $H_y$ and vanishes outside $H_y$.
Then Liouville Theorem yields the existence of a constant $\mathbf{a}_y\ge 0$ so that 
$$F_0=\mathbf{a}_y(\nu_y\cdot z)_+.$$ 
On the other hand, using again that $y\in \partial^*\Omega$, it follows that  $\mathscr{H}^{n-1}|_{\partial^* \Omega_{y,r}} \rightharpoonup \mathscr{H}^{n-1}|_{\partial H_y}$.
Hence $\Delta F_0=\az (y)\mathscr{H}^{n-1}|_{\partial H_y}$. Combining these two consequences, we conclude that $\az (y)=\mathbf{a}_y$.
This shows that $F_{r_k}\to \mathbf{a}_y(\nu_y\cdot z)_+$ for any converging subsequence $r_k$, so the entire sequence $F_r$ converges to $\mathbf{a}_y(\nu_y\cdot z)_+$, as desired.
\end{proof}

We can now prove a maximum principle for $|\nabla u|$.
\begin{prop}\label{max nabla u}
    Let $\Omega$ be an open bounded set satisfying the assumptions in Theorem~\ref{main thm}, and let $u$ solve \eqref{eq:weak set finite per}. Then
    $$\sup_{\Omega} |\nabla u|\le \mathbf{c}. $$
\end{prop}
\begin{proof}
For $\ez>0$ and $e \in \mathbb S^{n-1}$, let
$$\phi_\ez(y)=\frac{G_x(y+\ez e)-G_x(y-\ez e)}{2\ez},$$
and note that $\phi_\ez$ is uniformly bounded and Lipschitz continuous near $\partial\Omega$.
Thus, testing \eqref{eq:weak set finite per} against $\phi_\ez$ we obtain 
\begin{align}
\mathbf{c} \int_{\partial^*\Omega}  \phi_\ez \, d\mathscr H^{n-1} -\int_{\Omega} \phi_\ez\, dy =\int_{\Omega} u \Delta  \phi_\ez \, dy. \label{test phi}  
 \end{align}
Recall that, by Lemma~\ref{greens function}(2), we have  
$$
    \Delta \phi_\ez(y) =   \frac{  \az \, \mathscr H^{n-1}|_{\partial^*\Omega-\ez e} - \az  \,  \mathscr H^{n-1}|_{\partial^*\Omega+\ez e}}{2\ez} 
    -\frac{\delta_{x- \ez e} - \delta_{x+\ez e}}{2\ez}.    
$$
Therefore, by a change of variable and Lemma~\ref{basic property u}(4) we get
\begin{align*}
     \int_{\Omega} u\Delta \phi_\ez \, dx
   &=     \frac{1}{2\ez}\left[ \int_{\partial^*\Omega-\ez e} \az(y+\ez e) u(y)\, d\mathscr H^{n-1} -  \int_{\partial^*\Omega+\ez e} \az(y-\ez e) u(y)\, d\mathscr H^{n-1} \right]  \\
   &\quad \quad -\left\langle u,\, \frac{\delta_{x- \ez e} - \delta_{x+\ez e}}{2\ez}\right\rangle\\
   &=  -\int_{\partial^*\Omega} \frac{u(y+\ez e)-u(y-\ez e)}{2\ez} \az(y)\, d\mathscr H^{n-1} + \left\langle \frac {u(y+ \ez e) - u(y-\ez e)}{2\ez} ,\, \delta_x \right\rangle \\
   &= -   \frac{ \mathbf{c}}{2}  \int_{\partial^*\Omega}  \az \nu_y \cdot e \, d\mathscr H^{n-1}  + \partial_e u(x) + o(1).
\end{align*}
Concerning the left-hand side of \eqref{test phi}, we note that
$\phi_\ez\to DG\cdot e$ inside $\Omega$. Thus, by dominated convergence we get
$$\int_{\Omega} \phi_\ez\, dy=\int_{\Omega} \partial_e G_x\, dy +o(1).$$
In addition, Lemma~\ref{greens function} tells that, 
 for $\mathscr H^{n-1}$-almost every $y\in \partial^*\Omega$, we have
$$
\phi_\ez(y)=\frac{G_x(y+\ez e)-G_x(y-\ez e)}{2\ez}
\to \frac{\mathbf{a}_y}2 \big((\nu_y\cdot e)_+ - (\nu_y\cdot e)_-\big)=\frac{\az(y)}{2} \nu_y\cdot e.
$$
Thus, the left-hand side of \eqref{test phi} is equal to
$$ \frac{\mathbf{c}}{2} \int_{\partial^*\Omega} \az\nu_y\cdot e\, d\mathscr H^{n-1}- \int_{\Omega} \partial_e G_x\, dy +o(1). $$
As a result,  by letting  $\ez\to 0$ in \eqref{test phi} we eventually obtain
\begin{equation}
\label{eq:pe u}
    \partial_e u(x) =  {\mathbf{c} }  \int_{\partial^*\Omega} \az\nu_y \cdot e \, d\mathscr H^{n-1} -  \int_{\Omega} \partial_e G_x\, dy.
\end{equation}
Note now that, by the divergence theorem in sets for finite perimeter (recall that $\nu_y$ is the inner unit normal), 
\begin{equation}
\label{eq:pe G}
\int_{\Omega} \partial_e G_x\, dy = -\int_{\partial^* \Omega} G_x \nu_y\cdot e \, d\mathscr H^{n-1}= 0.
\end{equation}
Also, thanks to Lemma~\ref{greens function},
\begin{equation}
\label{eq:pe G 2}
0 = \int_{\mathbb R^n} \Delta G_x\, dy
= - 1 + \int_{\partial^*\Omega} \az \, d\mathscr H^{n-1}.
\end{equation}
Combining \eqref{eq:pe u}, \eqref{eq:pe G}, and \eqref{eq:pe G 2}, we get
$$|\partial_e  u(x)|\le \mathbf{c}\int_{\partial^*\Omega} \az \, d\mathscr H^{n-1}=\mathbf{c}.$$
Since $x$ and $e$ are arbitrary, the result follows.
\end{proof}

\begin{proof}[Proof of Theorem~\ref{main thm}]
Thanks to the previous results, we can basically repeat Weinberger's argument \cite{W1971} with minor modifications.

More precisely, recalling that we can assume $\Omega$ to be open (see Lemma~\ref{basic property u}(3)), define 
$$P=|\nabla u|^2 + \frac 2 n u \qquad \text{inside }\Omega.$$
Since $\Delta u=-1$ inside $\Omega$, it follows that
\begin{equation}
\label{eq:Delta P}
\Delta P = 2 |D^2 u|^2 - \frac 2 n \Delta u=2\biggl|{D^2u}+\frac{1}{n}\textrm{Id}\biggr|^2\ge 0. \end{equation}
In particular, as $u$ is smooth in $\{u> \eta\}$, Proposition~\ref{max nabla u} and the weak maximum principle imply that
$$
\max_{\{u\geq \eta\}}P=\max_{\partial\{u\geq \eta\}}P
\leq \mathbf c^2+\frac{2}n\eta \qquad \forall\,\eta>0,
$$
so letting $\eta \to 0$ we deduce that 
\begin{equation}
\label{eq:max P}
P\le \mathbf  c^2\qquad \text{in $\Omega$.}
\end{equation}
Using again that $\Delta u=-1$ inside $\Omega$, thanks to \eqref{volume} and \eqref{eq:max P}  we have
$$\mathbf{c}^2 |\Omega|=\frac{n+2}{n} \int_{\Omega} u  \, dx
=\int_{\Omega} u\biggl(-\Delta u +\frac{2}{n} \biggr) \, dx=\int_{\Omega} \left( |\nabla u|^2 + \frac 2 n u \right) \, dx =\int_{\Omega} P \, dx\leq  \mathbf{c}^2 |\Omega|.$$
The equation above implies that $P= \mathbf c^2$ inside $\Omega$. In particular $\Delta P=0$, so it follows from 
\eqref{eq:Delta P} that 
$$
{D^2u}= -\frac{1}{n}\textrm{Id}\qquad \text{in $\Omega$.}
$$
As $\Omega$ is indecomposable, this implies that $\Omega$ is a ball and that, up to a translation, $u$ is given by \eqref{eq:sol u}.
\end{proof}

\section{Extension to slit domains}
\label{sect:extension}

In the previous section we proved the validity of Serrin's Theorem in the setting of sets of finite perimeter.
However, as already mentioned in Remark~\ref{rem:slit}, this formalism does not allow one to treat sets with slit discontinuities, say a slit ball. In fact,  from a measure-theoretic point of view, a slit ball is equivalent to a ball.

To extend our result to slit domains, we need to find a suitable reformulation of \eqref{serrin}. To this end, we note that in the proof of Theorem~\ref{main thm}, the Neumann condition on $u$ was crucially used to prove the blow-up result in Lemma~\ref{basic property u}(4).
In the case of slit domains, the assumption that $\Delta u=2\mathbf{c}\mathscr{H}^{n-1}$ on the slit does not ensure that on both sides of the slit the function $u$ behaves like a linear function with slope $\mathbf c$ (the slopes on the two sides may be different). 
As we will see, to prove Lemma~\ref{lem:volume} in the case of a slit domain, the equality of the slopes from the two sides of the slit is required. In addition, we need to guarantee that $u$ vanishes on the slit in a suitable weak sense.
Both conditions are included in the assumption \eqref{eq:blow up1} below.

The following generalization of the Theorem~\ref{main thm} holds.

\begin{thm}\label{main thm 2}
Let $\Omega\subset \mathbb R^n$ be a indecomposable set  of finite perimeter,
let $u \in W^{1,2}(\R^n)$ satisfy
\begin{equation}
\label{eq:weak Lip}
u=0\quad \text{a.e. in }\R^n\setminus \Omega,\qquad \Delta u=\mathbf{c}\mathscr{H}^{n-1}|_{\partial^*\Omega}
+2\mathbf{c}\mathscr{H}^{n-1}|_{\Sigma}- \mathbf{1}_{\Omega}\,dx
\end{equation}
for some $(n-1)$-rectifiable set $\Sigma\subset \mathring\Omega$, and assume that 
\begin{equation}\label{domain boundary2}
\mathscr H^{n-1}\big(B_r(x)\cap (\partial^* \Omega\cup \Sigma)\big)\le A r^{n-1} \quad \text{ for $\mathscr H^{n-1}$-a.e. } x\in \partial^* \Omega\cup \Sigma \text{ and $r \in (0,1)$.} 
\end{equation}
Also, suppose that
at $\mathscr H^{n-1}$-a.e. $x \in \Sigma$ it holds
\begin{equation}
\label{eq:blow up1}
\frac{u(x+rz)}{r} \to \mathbf{c}|\nu_x\cdot z| \quad \text{ as }r\to 0,
\end{equation}
where $\nu_x$ denotes the measure-theoretic unit normal at $x$. 
Then, up to a translation, 
 $\Omega$ is a ball of radius $R=R(n,\mathbf{c})>0$ and $u$ is given by \eqref{eq:sol u}. 
\end{thm}
To prove this theorem, we first note that the following generalization of Lemma~\ref{basic property u} holds with essentially the same proof.

\begin{lem}\label{basic property u2}
    Let $\Omega$ and $u$ satisfy the assumption of Theorem~\ref{main thm 2}.
Then:
\begin{enumerate}
    \item[(1)] $u$ is $L$-Lipschitz continuous, with $L=L(n,\mathbf{c},A)$.
    \item[(2)] $u$ is nonnegative,  $\{u>0\}=\mathring\Omega\setminus\overline\Sigma$, and $u \in C^\infty(\mathring\Omega\setminus\overline\Sigma)$.
     \item[(3)] $\Omega=\mathring\Omega\setminus\overline\Sigma$ up to a set of measure zero. 
\item[(4)] At every point $x\in \partial^*\Omega$ it holds
$$
\frac{u(x+rz)}{r} \to \mathbf{c}\big(\nu_x\cdot z \big)_+ \quad \text{ as }r\to 0,
$$
where $\nu_x$ denotes the measure-theoretic inner unit normal at $x$. 
\end{enumerate}
\end{lem}

We then show that also Lemma~\ref{lem:volume} holds.

\begin{lem}\label{lem:volume 2}
Let $\Omega$ and $u$ satisfy the assumptions in Theorem~\ref{main thm 2}. Then 
$$
    (n+2)\int_{\Omega} u\, dx =\mathbf{c}^2 n |\Omega|. 
$$
\end{lem}
\begin{proof}
As in the proof of Lemma~\ref{lem:volume} we consider the function $\varphi_\ez$ and $\psi_\ez$ defined as in \eqref{eq:varphi eps} and \eqref{eq:psi eps}.
Testing  \eqref{eq:weak Lip} against $\varphi_\ez$ we get
\begin{equation}
    \label{eq:test vol3}
\mathbf{c} \int_{\partial^*\Omega}  \varphi_\ez \, d\mathscr H^{n-1} + 2 \mathbf{c} \int_{\Sigma}  \varphi_\ez \, d\mathscr H^{n-1} -\int_{\Omega} \varphi_\ez\, dx
= \int_{\Omega} u \Delta  \varphi_\ez \, dx.  
\end{equation}
We only need to treat the second and last term above, since the others are treated as in the proof of Lemma~\ref{lem:volume}.

In this case it holds
\begin{align*}
    \Delta   \varphi_\ez(x)&=\mathbf{c} \frac{(1+\ez)^2\mathscr H^{n-1}|_{(1+\ez)^{-1}\partial^*\Omega} - (1-\ez)^2\mathscr H^{n-1}|_{(1-\ez)^{-1}\partial^*\Omega}}{2\ez} \\
    &+2\mathbf{c} \frac{(1+\ez)^2\mathscr H^{n-1}|_{(1+\ez)^{-1}\Sigma} - (1-\ez)^2\mathscr H^{n-1}|_{(1-\ez)^{-1}\Sigma}}{2\ez} \\
    &-2\mathbf{c}\big(\mathscr H^{n-1}|_{\partial^*\Omega}+2\mathscr H^{n-1}|_{\Sigma}\big) +O\biggl(\frac {1} {\ez} dx  |_{\mathcal N_\ez} \biggr),
\end{align*} 
therefore (cp. \eqref{varphi eps})
\begin{multline*}
    \int_\Omega u \Delta   \varphi_\ez\,dx
= \mathbf{c}\left[\int_{ \partial^*\Omega}   \frac{(1+\ez)^{3-n} u((1+\ez)^{-1} x) -(1-\ez)^{3-n}  u( (1-\ez)^{-1} x)}{2\ez} \,d \mathscr H^{n-1}(x)\right]\\
+2\mathbf{c}\left[\int_{ \Sigma}   \frac{(1+\ez)^{3-n} u((1+\ez)^{-1} x) -(1-\ez)^{3-n}  u( (1-\ez)^{-1} x)}{2\ez} \,d \mathscr H^{n-1}(x)\right]
+\int_{\mathcal N_\ez} u\,O\biggl(\frac {1} {\ez}\biggr) \, dx.
    \end{multline*}
Hence, recalling \eqref{eq:test vol3} we deduce that (cp. \eqref{eq:test vol2})
\begin{equation}
    \label{eq:test vol4}
 \mathbf{c} \int_{\Omega^*}  \psi_\ez \, d\mathscr H^{n-1} + 2\mathbf{c} \int_{\Sigma }  \psi_\ez \, d\mathscr H^{n-1} -\int_{\Omega} \varphi_\ez\, dx
=  \int_{\mathcal N_\ez} u\,O\biggl(\frac {1} {\ez}\biggr) \, dx.
\end{equation}
Note now that applying Lemma~\ref{basic property u2}(4),
for $\mathscr H^{n-1}$-almost every $x\in \partial^*\Omega$ we have
$$
\psi_\ez(x)
\to -\mathbf{c} \big((\nu_x\cdot x)_+ - (\nu_x\cdot x)_-\big)=-\mathbf{c}\nu_x\cdot x, 
$$
while \eqref{eq:blow up1} implies that, for $\mathscr H^{n-1}$-almost every $x\in \Sigma$,
$$\psi_\ez(x) \to \mathbf{c} \big((\nu_x\cdot x)_+ - (\nu_x\cdot x)_-\big)
+\mathbf{c} \big((\nu_x\cdot x)_- - (\nu_x\cdot x)_+\big)=0.$$
Thus, by dominated convergence,
$$
 \mathbf{c} \int_{\Omega^*}  \psi_\ez \, d\mathscr H^{n-1} + 2\mathbf{c} \int_{\Sigma }  \psi_\ez \, d\mathscr H^{n-1} \to -
\mathbf{c}^2 \int_{\partial^*\Omega} \nu_x\cdot x \, d\mathscr H^{n-1}\qquad \text{as }\ez \to 0.
$$
Combining this fact with \eqref{eq:test vol4}, we conclude as in the proof of Lemma~\ref{lem:volume}.
\end{proof}

The next step is to prove a maximum principle for $
|\nabla u|$. In this case, given a point $x \in \mathring \Omega\setminus \overline\Sigma$, we define the Green function $G_x$ by considering an increasing sequence $\Omega_k$ of smooth sets contained inside $\mathring \Omega\setminus \overline\Sigma$ such that $\Omega_k \to \mathring \Omega\setminus \overline\Sigma$ as $k \to \infty$. In this way, the following analog of Lemma~\ref{greens function} holds.
\begin{lem}\label{greens function2}
    Let $G_x$ be the Green function constructed above. Then:
    \begin{enumerate}
    \item[(1)] $G_x$ is Lipschitz continuous near $\partial \Omega \cup \overline \Sigma$, with the Lipschitz constant depending only on $n$, $A$, $\mathbf{c}$, $\Omega$, $\Sigma$, and $x$. Moreover, there exist bounded measurable functions $\az:\partial^*\Omega \to [0,\infty)$ and $\beta:\Sigma \to [0,\infty)$  such that
    $$\Delta G_x= \az\,\mathscr H^{n-1}|_{\partial^*\Omega}+\beta\,\mathscr H^{n-1}|_{\Sigma}-\delta_x.$$     
    \item[(2)]  At every point $y\in \partial^*\Omega$ it holds
$$
\frac{G_x(y+rz)}{r} \to \mathbf{a}_y\big(\nu_y\cdot z \big)_+ \quad \text{ as }r\to 0,
$$
while at every point $y\in \Sigma$ it holds
$$
\frac{G_x(y+rz)}{r} \to \mathbf{a}_y^+\big(\nu_y\cdot z \big)_++\mathbf{a}_y^-\big(\nu_y\cdot z \big)_- \quad \text{ as }r\to 0,
$$
where $\nu_y$ denotes the measure-theoretic inner unit normal at $y$, $\mathbf{a}_y=\az(y)$, $\mathbf{a}_y^\pm \geq 0$, and $\mathbf{a}_y^+ + \mathbf{a}_y^-=\beta(y)$. 
    \end{enumerate}
\end{lem}
We can now prove the maximum principle for $|\nabla u|$.

\begin{prop}\label{max Du 2}
    Let $\Omega$ and $u$ satisfy the assumptions in Theorem~\ref{main thm 2}. Then
    $$\sup_{\mathring \Omega\setminus \overline\Sigma} |\nabla u|\le \mathbf{c}. $$
\end{prop}
\begin{proof}
As in the proof of Proposition~\ref{max nabla u}, we test \eqref{eq:weak Lip} against 
$\phi_\ez(y)=\frac{G_x(y+\ez e)-G_x(y-\ez e)}{2\ez}$
to get
\begin{equation}
 \mathbf{c} \int_{\partial^*\Omega}  \phi_\ez \, d\mathscr H^{n-1} + 2 \mathbf{c}  \int_{\Sigma}  \phi_\ez \, d\mathscr H^{n-1} -\int_{\Omega} \phi_\ez\, dy
= \int_{\Omega} u \Delta  \phi_\ez \, dy. \label{test phi 2} 
 \end{equation}
Recalling Lemma~\ref{greens function2}(1) we have  
\begin{multline*}
    \Delta \phi_\ez(y) =   \frac{  \az  \, \mathscr H^{n-1}|_{\partial^*\Omega-\ez e} - \az   \,  \mathscr H^{n-1}|_{\partial^*\Omega+\ez e}}{2\ez} +\frac{   \beta\, \mathscr H^{n-1}|_{\Sigma-\ez e} - \beta \,  \mathscr H^{n-1}|_{\Sigma+\ez e}}{2\ez} \\
    -\frac{\delta_{x- \ez e} - \delta_{x+\ez e}}{2\ez}.    
\end{multline*}
Hence, by a change of variable, Lemma~\ref{basic property u2}(4), and \eqref{eq:blow up1}, we now get
\begin{align*}
     \int_{\Omega} u\Delta \phi_\ez \, dx
   &=    \frac{1}{2\ez}\left[ \int_{\partial^*\Omega-\ez e} \az(y+\ez e) u(y)\, d\mathscr H^{n-1} -  \int_{\partial^*\Omega+\ez e} \az(y-\ez e) u(y)\, d\mathscr H^{n-1} \right]  \\
   &+\frac{1}{2\ez}\left[ \int_{\Sigma - \ez e} \beta(y+\ez e) u(y)\, d\mathscr H^{n-1} -  \int_{\Sigma+\ez e} \beta(y-\ez e) u(y)\, d\mathscr H^{n-1} \right]  \\
   &\quad \quad -\left\langle u,\, \frac{\delta_{x - \ez e} - \delta_{x+\ez e}}{2\ez}\right\rangle\\
   &= - \left[ \int_{\partial^*\Omega} \frac{u(y+\ez e)-u(y-\ez e)}{2\ez} \az(y)\, d\mathscr H^{n-1} \right]\\
   & -  \left[ \int_{\Sigma} \frac{u(y+\ez e)-u(y-\ez e)}{2\ez} \beta(y)\, d\mathscr H^{n-1} \right] + \left\langle \frac {u(x+ \ez e) - u(x-\ez e)}{2\ez} ,\, \delta_x \right\rangle \\
   &= -   \frac{ \mathbf{c}}{2}  \int_{\partial^*\Omega}  \az  \nu_y \cdot e \, d\mathscr H^{n-1}  + \partial_e u(x) + o(1)
\end{align*}
For the left-hand side of \eqref{test phi 2}, we have
$$\int_{\Omega} \phi_\ez\, dy=\int_{\Omega} \partial_e G_x\, dy +o(1).$$
Moreover, Lemma~\ref{greens function2}(2) implies that, 
 for $\mathscr H^{n-1}$-almost every $y\in \partial^*\Omega$ we have
$$
\phi_\ez(y)=\frac{G_x(y+\ez e)-G_x(y-\ez e)}{2\ez}
\to \frac{\mathbf{a}_y}2 \big((\nu_y\cdot e)_+ - (\nu_y\cdot e)_-\big)=\frac{\az (y)}{2} \nu_y\cdot e, 
$$
while for $\mathscr H^{n-1}$-almost every $y\in \Sigma$ we have
$$\phi_\ez(y) \to  \frac{\mathbf{a}_{y}^+} 2  \big((\nu_y\cdot e)_+ - (\nu_y\cdot e)_-\big)
+\frac {\mathbf{a}_{y}^-} 2 \big((\nu_y\cdot e)_- - (\nu_y\cdot e)_+\big)=\frac{{\bf a}_{y}^+-{\bf a}_{y}^-} 2(\nu_y\cdot e).$$
Thus the left-hand side of \eqref{test phi 2} is equal to
$$ \frac{\mathbf{c}}{2} \int_{\partial^*\Omega} \az \nu_y\cdot e\, d\mathscr H^{n-1}+  {\mathbf{c}}  \int_{\Sigma}  ({\bf a}_{y}^+-{\bf a}_{y}^-) \nu_y\cdot e\, d\mathscr H^{n-1} - \int_{\Omega} \partial_e G_x\, dy +o(1). $$
As a result,  by letting  $\ez\to 0$ in \eqref{test phi 2} we eventually arrive at
$$\partial_e u(x) =  {\mathbf{c} }  \int_{\partial^*\Omega} \az \nu_y\cdot e \, d\mathscr H^{n-1} +  {\mathbf{c} }  \int_{\Sigma} ({\bf a}_{y}^+-{\bf a}_{y}^-)\nu_y\cdot e \, d\mathscr H^{n-1} -  \int_{\Omega} \partial_e G_x\, dy.$$
As before, the divergence theorem implies
$$\int_{\Omega} \partial_e G_x\, dy = -\int_{\partial \Omega} G_x \nu_y\cdot e \, d\mathscr H^{n-1}= 0,$$
while Lemma~\ref{greens function} yields
\begin{multline*}
0 = \int_{\mathbb R^n} \Delta G_x\, dy 
= - 1 + \int_{\partial^*\Omega} \az  \, d\mathscr H^{n-1} + \int_{\Sigma} \beta \, d\mathscr H^{n-1}\\
\ge  - 1 + \int_{\partial^*\Omega} \az  \, d\mathscr H^{n-1} + \int_{\Sigma} |{\bf a}_{y}^+-{\bf a}_{y}^-| \, d\mathscr H^{n-1}. 
\end{multline*}
where we used that $|{\bf a}_{y}^+-{\bf a}_{y}^-|\le {\bf a}_{y}^++{\bf a}_{y}^-= \beta(y)$.
This proves that $|\partial_e u(x)|\leq \mathbf{c}$, and we conclude by the arbitrariness of $x$ and $e$. 
\end{proof}
 
Thanks to these preliminary results, we can now repeat the argument in the proof of Theorem~\ref{main thm} (with the only difference that $P$ is now defined inside $\mathring\Omega\setminus\overline\Sigma$) to conclude the validity of Theorem~\ref{main thm 2}.

\end{document}